\documentclass[A4, 12pt]{amsart}
\usepackage{amssymb, amsmath, color, tikz, hyperref, cleveref}
\usepackage[T1]{fontenc}
\usepackage{mathptmx}
\usepackage{todonotes, comment}
\usepackage{enumerate}
\usepackage{array}
\usepackage{multirow}
\usepackage{graphicx, subcaption}
\usepackage[margin=2cm]{geometry}
\usepackage[inline]{enumitem}

\newcommand\captimes{\mathbin{\ooalign{$\cap$\cr%
   \hfil\raise0.42ex\hbox{$\scriptscriptstyle\times$}\hfil\cr}}}
\newcommand\bigcaptimes{\mathop{\ooalign{$\bigcap$\cr%
 \hfil\raise0.36ex\hbox{$\scriptscriptstyle\boldsymbol{\times}$}\hfil\cr}}}

\newtheorem{theorem}{Theorem}
\numberwithin{theorem}{section}

\newtheorem{lemma}[theorem]{Lemma}
\newtheorem{corollary}[theorem]{Corollary}

\theoremstyle{definition}

\newtheorem{remark}[theorem]{Remark}

\crefname{question}{question}{questions}

\title{Fractional discrete Helly for pairs in a family of boxes}
\author{Taehyun Eom \and Minki Kim \and  Eon Lee}
\address[Taehyun Eom]{R\&BD Foundations, Chonnam National University, Gwangju, South Korea}
\email{taeheom@chonnam.ac.kr}
\address[Minki Kim]{Department of Mathematical Sciences, Gwangju Institute of Science and Technology, Gwangju, South Korea}
\email{minkikim@gist.ac.kr}
\address[Eon Lee]{School of Electrical Engineering and
Computer Science, Gwangju Institute of Science and Technology, Gwangju, South Korea}
\email{eonlee27@gm.gist.ac.kr}

\date{\today}

\begin{document}
\begin{abstract}
    Given a point set $S$ in $\mathbb{R}^d$, a family of sets is $S$-intersecting if its members have a point in common in $S$.
    Recently, Edwards and Sober\'{o}n proved a fractional version of Halman's theorem for axis-parallel boxes, showing that every finite family $F$ of axis-parallel boxes in $\mathbb{R}^d$ with positive density of $S$-intersecting $(d+1)$-tuples contains an $S$-intersecting subfamily of size linear in $|F|$.
    We prove that qualitatively the same conclusion can be achieved if the density of $S$-intersecting pairs is sufficiently large.
\end{abstract}

\thanks{Corresponding author: Minki Kim (\texttt{minkikim@gist.ac.kr}).}

\subjclass[2020]{52A35}

\keywords{fractional S-Helly property, axis-parallel boxes}

\maketitle
    
\section{Introduction}
Throughout this paper, we consider closed boxes only.
It was shown by Eckhoff~\cite{Eck88} that for each positive integer $d$ and real number $\alpha \in (1-1/d,1]$, every finite family $F$ of axis-parallel boxes with at least $\alpha\binom{|F|}{2}$ many intersecting pairs contains an intersecting subfamily of size at least $(1-\sqrt{d(1-\alpha)})|F|$.
We establish a discrete version of Eckhoff's result.

Given a nonempty set $S$ and a family $F$ of nonempty sets, we say $F$ is {\em $S$-intersecting} if the intersection of all members of $F$ meets $S$.
Our main theorem states the following:
\begin{theorem}\label{thm:main}
    For every positive integer $d$, there is a number $c_d\in(0,1)$ and a function $\beta_d:(c_d,1]\to(0,1]$ such that the following holds:
    for every $S \subset\mathbb{R}^d$, $\alpha\in(c_d,1]$, and finite family $F$ of axis-parallel boxes in $\mathbb{R}^d$, if there are at least $\alpha\binom{|F|}{2}$ of $S$-intersecting pairs of $F$, then $F$ contains an $S$-intersecting subfamily of size at least $\beta_d(\alpha)|F|$.
\end{theorem}

Helly's theorem for axis-parallel boxes states that if every two boxes meet, then all boxes have a point in common.
A discrete version of Helly's theorem, or the {\em $S$-Helly theorem}, for axis-parallel boxes was introduced in 2008 by Halman~\cite{Hal08}. It asserts that for every finite family $F$ of axis-parallel boxes in $\mathbb{R}^d$ and $S \subset \mathbb{R}^d$, if every $2d$-tuple is $S$-intersecting, then $F$ is $S$-intersecting. See, for example, \cite[Section 2.1]{ALS17} for brief overview on $S$-Helly theorems for other set systems.

Recently, Edwards and Sober\'{o}n~\cite{ES24+} established colorful and fractional generalizations of Halman's theorem, by showing the $(2d-1)$-collapsibility of the simplicial complex on $F$ whose simplices are precisely the $S$-intersecting subfamilies. From the $(2d-1)$-collapsibility of such a complex, it immediately follows from a result by Kalai~\cite{Kal84} that given $S \subset \mathbb{R}^d$ and $\alpha\in(0,1]$, every finite family $F$ of axis-parallel boxes in $\mathbb{R}^d$ with at least $\alpha\binom{|F|}{2d}$ many $S$-intersecting $2d$-tuples contains an $S$-intersecting subfamily of size $(1-(1-\alpha)^{1/2d})|F|$.
In addition, Edwards and Sober\'{o}n proved that if the density of $S$-intersecting $(d+1)$-tuple is positive, then the density of $S$-intersecting $2d$-tuples is also positive. This implies that axis-parallel boxes have a fractional $S$-Helly property for $(d+1)$-tuples.
\begin{theorem}[\cite{ES24+}]\label{thm:d+1}
    For every positive integer $d$, there is a function $\gamma_d:(0,1]\to(0,1]$ such that the following holds:
    for every $\alpha\in(0,1]$ and finite family $F$ of axis-parallel boxes in $\mathbb{R}^d$, if there are at least $\alpha\binom{|F|}{d+1}$ many $S$-intersecting $(d+1)$-tuples of $F$, then $F$ contains an $S$-intersecting subfamily of size at least $\gamma_d(\alpha)|F|$. 
\end{theorem}
Theorem~\ref{thm:main} extends the above theorem to a fractional $S$-Helly property for pairs under certain conditions. This tells us that, even though the Helly number and the $S$-Helly number for axis-parallel boxes are different, we have qualitatively similar fractional properties.

In order to prove Theorem~\ref{thm:main}, it is sufficient to show that the following holds.
\begin{theorem}\label{thm:main2}
    There exists a function $N:\mathbb{N}\to\mathbb{N}$ such that for every $S\subset\mathbb{R}^d$ and family $F$ of axis-parallel boxes in $\mathbb{R}^d$ with $|F|=N(d)$, if every pair of $F$ is $S$-intersecting, then $F$ contains an $S$-intersecting subfamily of size $d+1$.
\end{theorem}
Once the above statement is confirmed, the supersaturation phenomenon of graphs allows us to apply Theorem~\ref{thm:d+1} under the conditions in Theorem~\ref{thm:main}.
Here we state a specific statement that we will use.
\begin{theorem}[supersaturation]\label{thm:turan}
    Given a positive integer $d\geq2$, there exist a constant $t_m\in(0,1)$ and a function $f_m:(t_m,1]\to(0,1]$ such that for every $\alpha \in (t_m,1]$, every simple undirected graph on $n$ vertices with at least $\alpha\binom{n}{2}$ many edges contains at least $f_{m}(\alpha)\binom{n}{m}$ many copies of $K_m$.
\end{theorem}
\begin{proof}[Proof of Theorem~\ref{thm:main} from Theorem~\ref{thm:main2}]
We show that the statement is true if we set 
$$c_d = t_{N(d)}\text{ and }\beta_d(\alpha) = \gamma_d\left(\left(f_{N(d)}(\alpha)/\binom{N(d)}{d+1}\right)\right).$$
Let $F$ be a finite family of $n$ axis-parallel boxes in $\mathbb{R}^d$ with at least $\alpha\binom{n}{2}$ many $S$-intersecting pairs for some $\alpha \in (c_d,1]$.
Consider a graph on $F$ whose edges are the $S$-intersecting pairs in $F$. Then Theorem~\ref{thm:turan} implies that there are $f_{N(d)}(\alpha)\binom{n}{N(d)}$ many cliques of size $N(d)$.
By Theorem~\ref{thm:main2}, each such clique contains a subset of size $d+1$ that corresponds to an $S$-intersecting $(d+1)$-tuple of $F$.
Since each $(d+1)$-tuple can be contained in at most $\binom{n-d-1}{N(d)-d-1}$ many $N(d)$-tuples, we see that $F$ contains at least $\alpha'\binom{n}{d+1}$ many $S$-intersecting $(d+1)$-tuples, where $\alpha'=\left(f_{N(d)}(\alpha)/\binom{N(d)}{d+1}\right)$.
Finally, Theorem~\ref{thm:d+1} shows that $F$ contains an $S$-intersecting subfamily of size at least $\gamma_d(\alpha')n$.
\end{proof}

The existence of $N(d)$ in Theorem~\ref{thm:main2} will be shown in Section~\ref{section:main}. In Section~\ref{sec:further}, we discuss an analogue of Theorem~\ref{thm:main2} that replaces the pairwise intersection pattern with $d$-tuplewise intersection pattern.

\section{Intersection patterns of pairwise $S$-intersecting axis-parallel boxes}\label{section:main}
In this section, we present a proof of Theorem~\ref{thm:main2}.
Actually, we will prove a slightly stronger statement: for every finite family $F=\{B_1,\ldots,B_n\}$ of axis-parallel boxes in $\mathbb{R}^d$ with sufficiently large $n$, if every pair of $F$ intersects, there exists $i_1,\ldots,i_{d+1}\in[n]$ such that $B_{i_d} \cap B_{i_{d+1}} = B_{i_1}\cap B_{i_2}\cap\cdots\cap B_{i_{d+1}}$.
It immediately follows that any point that meets $B_{i_d}\cap B_{{i_{d+1}}}$ is contained in $B_{i_1}\cap B_{i_2}\cap\cdots\cap B_{i_{d+1}}$.

For a finite family $F$ of axis-parallel boxes in $\mathbb{R}^d$, let $$\pi_i(F) = \{\pi_i(A)\mid A\in F\} = \{I_j^{(i)}=[a_j^{(i)},b_j^{(i)}]\mid j\in[n]\},$$ where $\pi_i$ is the projection map onto the $i$-th coordinate.
Here, all boxes can be assumed to be bounded due to the finiteness of $F$.
If every $b$-tuple of $F$ is $S$-intersecting, then each $\pi_i(F)$ is an intersecting family of intervals by Helly's theorem for axis-parallel boxes. Consider two permutations on $[n]$ arising from the arrangement of endpoints of the intervals in $\pi_i(F)$. That is, we take permutations $\ell_i, r_i \in S_n$ such that 
$$\ell_i(j) < \ell_i(k)\text{ only if }a_j^{(i)} \leq a_k^{(i)}\text{, and }r_i(j) < r_i(k)\text{ only if }b_j^{(i)} \geq b_k^{(i)}.$$ 
%See Figure~\ref{fig:permu} for an illustration of these permutations.

%\begin{figure}[htbp]
%    \centering
%    \includegraphics[scale=0.75]{ell and r.pdf}
%    \caption{An example when $d=2$ and $n=3$. Note that taking $\ell_1 = 132$ is also possible.}
%    \label{fig:permu}
%\end{figure}

Given a permutation $\sigma$, define a total order $\preceq_\sigma$ on \([n]\) by \(a\preceq_\sigma b\) if and only if \(\sigma(a)\leq\sigma(b)\), and let \(\max_\sigma\) be the maximum function with respect to \(\preceq_\sigma\).
We observe that for every $j,k,m\in[n]$,
$$m\in \{s\in[n]: s\leq \text{max}_{\ell_i}\{j, k\}\}\cap \{s\in[n]: s\leq \text{max}_{r_i}\{j, k\}\} \implies I_j^{(i)} \cap I_k^{(i)} \subset I_m^{(i)}.$$
This holds for all $i\in[d]$, so we immediately obtain the following lemma.
\begin{lemma}\label{lem:premain}
    For every \(P\subseteq [n]\) and \(P'=\bigcap_{i=1}^d\left(\{s\in[n]: s\preceq_{\ell_i}\max_{\ell_i}P\}\cap\{s\in[n]: s\preceq_{r_i}\max_{r_i}P\}\right)\),
    \[
        \bigcap_{j\in P}B_j=\bigcap_{j\in P'}B_j.
    \]
\end{lemma}
\begin{proof}
By the definition of $P'$, the observation in the above implies $\bigcap_{j\in P}B_j\subseteq\bigcap_{j\in P'}B_j$. The other direction $\bigcap_{j\in P}B_j\supseteq\bigcap_{j\in P'}B_j$ follows from an obvious observation that $P \subset P'$.
\end{proof}

We denote by $\Sigma_F$ the set of all such permutations, that is, \(\Sigma_F=\{\ell_1,r_1,\cdots,\ell_d,r_d \}\subseteq S_n\).

For every \(P\subseteq [n]\) and \(\sigma\in S_n\), let \(\sigma(\leq P)=\{s\in[n]: s\preceq_\sigma \max_\sigma P\}\). For each \(A\subseteq S_n\), we define 
\begin{itemize}
    \item the \emph{\(P\)-dependent set} of \(A\) as \(\displaystyle\langle A; P\rangle:=\bigcap_{\sigma\in A}\sigma(\leq P)\),
    \item the \emph{\(P\)-dependency} as \(\displaystyle d(A;P):=|\langle A;P\rangle|-|P|\), and
    \item  the \emph{$p$-tuple dependency} of \(A\) as \(\displaystyle d_p(A):=\max \{d(A;P): |P|=p\}\).
\end{itemize}
Now we reduce our problem into a purely combinatorial form with the following lemma.

\begin{lemma} \label{lem:main}
    Let \(S\) be a set in \(\mathbb{R}^d\), \(F=\{B_1,\cdots,B_n\}\) be a family of \(n\) axis-parallel boxes in \(\mathbb{R}^d\) and \(p\geq 2\) be a positive integer.
    If every \(p\)-tuple of \(F\) is \(S\)-intersecting, then \(F\) contains an \(S\)-intersecting (\(d_p(\Sigma_F)+p\))-tuple.
\end{lemma}
\begin{proof}
    If $S=\varnothing$, then the statement is obvious. So we may assume that $S$ is nonempty.
    By the Lemma~\ref{lem:premain}, for every \(P\in\binom{[n]}{p}\), it must be
    \[
        S\cap \bigcap_{j\in \langle \Sigma_F;P\rangle}B_j = S\cap \bigcap_{j\in P}B_j \not=\varnothing,
    \]
    and hence we can find an \(S\)-intersecting \(|\langle \Sigma_F;P\rangle|\)-tuple.
    By taking $P$ that gives maximum value of \(|\langle \Sigma_F;P\rangle|\), we ensure that $F$ contains an $S$-intersecting (\(d_p(\Sigma_F)+p\))-tuple.
\end{proof}

\subsection{Proof of Theorem~\ref{thm:main2}}\label{sec:mainpf}
By Lemma~\ref{lem:main}, it is sufficient to show that for every large enough $n$, we have $d_2(A) \geq d-1$ for all $A\subset S_n$ with $|A|\leq 2d$.
The existence of such $n$ can be obtained as a special case of the following.
\begin{theorem}\label{thm:f(n,a,2)}
    For every positive integers $a$ and $b$, there exists $n = n(a,b)$ such that for every \(A\subseteq S_{n}\) with \(|A|\leq a\), we have \(d_2(A)\geq b-1\).
\end{theorem}

Let \(\varnothing\not=A\subseteq S_ n\).
We may assume that \(A\) contains the identity permutation.
To see this, observe that for every \(\sigma,\tau\in S_n\) and \(P\subseteq [n]\), we have \((\sigma\tau^{-1})(\leq \tau (P))=\tau(\sigma(\leq P))\). By the definition, $i \in (\sigma\tau^{-1})(\leq \tau (P))$ if and only if $i=\tau(\sigma^{-1}(j))$ for some $j \leq \max_\sigma P$. 
On the other hand, the latter is also equivalent to $i=\tau(k)$ for some $k$ where $\sigma(k) \leq \max_\sigma P$, that is, $i\in \tau(\sigma(\leq P))$.
From this observation, we have \(\tau\langle A;P\rangle=\langle A\tau^{-1};\tau P\rangle\) and \(d(A; P)=d(A\tau^{-1};\tau P)\).
Therefore, it follows \(d_p(A)=d_p(A\tau^{-1})\), and hence we may replace $A$ with $A\sigma^{-1}$ for randomly chosen \(\sigma\in A\) which contains the identity permutation.

For any \(u,v\in [n]\) with \(u< v\), define the \emph{\(A\)-pattern} \[[u,v]_A=\{\sigma\in A\setminus\{\textrm{id}\}: u\prec_{\sigma}v\}.\]
Then we define \emph{pattern graph} of \(A\) as an edge-colored complete graph over the vertex set \([n]\) where for each edge \(uv\), its color is the \(A\)-pattern \([u,v]_A\), assuming that $u< v$ without loss of generality. Therefore, the pattern graph of $A$ has at most $2^{a-1}$ colors.

\begin{lemma}\label{lem:chain}
    If the pattern graph of $A$ contains a monochromatic path \(P=u_0u_1\cdots u_b\) with $u_i < u_j$ whenever $i<j$, then $d_2(A) \geq b-1$.
\end{lemma}
\begin{proof}
By the definition of the $A$-pattern, the assumption that the path $P$ is monochromatic implies that for each $\sigma\in A$, it must be either $u_0 \prec_\sigma u_1 \prec_\sigma \cdots \prec_\sigma u_b$ or $u_b \prec_\sigma \cdots \prec_\sigma u_1 \prec_\sigma u_0$.
In particular, this shows that $\langle A;\{u_0, u_b\}\rangle\supseteq\{u_0,\cdots,u_b\}$.
Thus we have \(d_2(A)\geq d(A;\{u_0, u_b\})\geq (b+1) - 2 = b-1\).
\end{proof}
\begin{proof}[Proof of Theorem~\ref{thm:f(n,a,2)}]
By the multicolor Ramsey theorem with $2^{a-1}$ colors, if $n$ is sufficiently large, then the pattern graph of \(A\) has a monochromatic clique of size $b+1$.
Since such a clique
It follows from Lemma~\ref{lem:chain} that \(d_2(A)\geq b-1\).
\end{proof}

This completes the proof of Theorem~\ref{thm:main2}.

\begin{remark}
    We also show that $N(d) = \Omega(d^{3/2})$ by constructing an example. Consider the standard basis $E=\{e_1,\ldots,e_d\}$ of $\mathbb{R}^d$ and take $d'=\lfloor\sqrt{d}\rfloor$ many subsets $E_1,\ldots,E_{d'}$ of $E$ such that every pair meets at exactly one element and no three have a common intersection.
    This can be found, for example, by taking $\lfloor\sqrt{d}\rfloor$ lines in general position in $\mathbb{R}^d$ and label the points where two lines meet. Then the set of points on a line corresponds to $E_i$.
    Now, let $B_i = \text{span}(E_i) \cap [-1,1]^d$, and take a family $F$ that consists of $\left\lfloor\frac{d}{2}\right\rfloor$ copies of each $B_i$. Finally, take $S = E$. Clearly, $F$ is a family of $\sqrt{d}\times\left\lfloor\frac{d}{2}\right\rfloor$ many axis-parallel boxes where every pair is $S$-intersecting but the largest $S$-intersecting subfamily has size $2\times\left\lfloor\frac{d}{2}\right\rfloor<d+1$.
\end{remark}

\subsection{Finding the optimal value of $N(2)$}\label{sec:N(2)}
In the previous subsection, we have shown an upper and a lower bounds on $N(d)$. However, there is a huge gap between them.
Here, we present the first step in finding the optimal value of $N(d)$.

Given a family $F$ of sets $B_1,B_2,\ldots,B_n$, the \emph{code} of $F$, denoted by $\textrm{code}(F)$, is the set of index subsets $J$ such that $\bigcap_{i\in J}B_i \setminus \bigcup_{i\notin J}B_i \neq \varnothing$. Each element of $\text{code}(F)$ is called a \emph{codeword}.
When each $B_i$ is an axis-parallel box, then it is known, for example, in \cite{BCH+24} that each codeword of $\text{code}(F)$ can be written as the intersection of two codewords where one is from $\text{code}(\pi_1(F))$ and the other is from $Y\in\text{code}(\pi_2(F))$.

In general, \(N(d)\geq n + 1\) is equivalent to the statement that there exists a family \(F=\{B_1, \cdots, B_n\}\) of pairwise intersecting axis-parallel boxes in \(\mathbb{R}^d\) such that for each \(i\not =k\), there exists a codeword $I_{j,k}$ of $\text{code}(F)$ of size at most $d$ that contains both $j$ and $k$. Note that if such a family exists, then by taking a set $S$ consists of one point from each nonempty region that corresponds to the codeword $I_{j,k}$, every pair of $F$ is $S$-intersecting but every maximal $S$-intersecting subfamily of $F$ has size at most $d$. In the plane, such a family with size $5$ cannot exist.
\begin{theorem}
    $N(2)=5$.
\end{theorem}
\begin{proof}
We want to show that for every family \(F=\{B_1, B_2, B_3, B_4, B_5\}\) of pairwise intersecting axis-parallel boxes in \(\mathbb{R}^2\), it must be \(\textrm{code}(F)\not\supseteq\binom{[5]}{2}\).
As observed at the beginning, for each $i\in[2]$, the arrangement of the intervals in $\pi_i(F)$ is determined by two permutations, $\ell_i$ and $r_i$.
This also implies that $\text{code}(\pi_i(F))$ consists of two chains one subchain from each chain $\omega^{i}$, $i\in[2]$ defined as
\begin{align*}
    &\omega^{i,1}: \varnothing\subseteq\{\ell_i^{-1}(1)\}\subseteq \{\ell_i^{-1}(1),\ell_i^{-1}(2)\}\subseteq\cdots\subseteq\{\ell_i^{-1}(1), \cdots,\ell_i^{-1}(5)\}=[5],\text{ and }\\
    &\omega^{i,2}: [5]=\{r_i^{-1}(1), \cdots,r_i^{-1}(5)\}\supseteq\cdots\supseteq\{r_i^{-1}(1)\} \supseteq\varnothing.
\end{align*}
For each chain \(\omega \in \{\omega^{i.j}:i,j\in[2]\} =: \Omega\), let \(\varnothing=\omega_0\subsetneq \omega_1\subsetneq\cdots\subsetneq \omega_5 = [5]\), so \(|\omega_i|=i\). 
Then we have
$\textrm{code}(F)\subseteq\{\varnothing\}\cup\{\omega_i^{1,p}\cap \omega_j^{2,q}\mid i, j\in[5],p,q\in [2]\}$.
Hence, for each \(C\in\mathrm{code}(F)\) with \(|C|=2\), it must be of one the following types: either \(C = \varphi_2\) or \(C = \varphi_3\cap \psi_3\) or \(C = \varphi_3\cap \psi_4\) for some \(\varphi,\psi\in\Omega\).
In any case, there exists a chain \(\varphi\in \Omega\) with \(C\subseteq\varphi_3\).

Now, assume \(\binom{[5]}{2}\subseteq\textrm{code}(F)\). Then
\[
    \binom{[5]}{2} = \bigcup_{\omega\in\Omega}\binom{\omega_3}{2}.
\]
We claim that there are two distinct $\varphi, \psi \in\Omega$ with $|\varphi_3\cap\psi_3|=2$.
First, if there exist two chains with \(\varphi_3=\psi_3\), then we have \[\left|\bigcup_{\omega\in\Omega}\binom{\omega_3}{2}\right|\leq\binom{3}{2}\times3=9<10=\binom{5}{2},\] which is a contradiction.
Therefore, we have $|\varphi_3\cap\psi_3|\leq 2$ for any choice.
We also note that \(|\varphi_3\cap \psi_3|\geq 1\). This is obvious because \(|\varphi_3\cup\psi_3| \leq |[5]|=5\).
Now, suppose \(|\varphi_3\cap \psi_3|=1\).
Then it must be \(\varphi_3\cup \psi_3=[5]\). Thus, by the pigeonhole principle, for every $\omega \in \Omega\setminus\{\varphi,\psi\}$, \(|\varphi_3\cap \omega_3|=2\) or \(|\psi_3\cap \omega_3|=2\).
In any case we can find two chains $\varphi, \psi \in\Omega$ with $|\varphi_3\cap\psi_3|=2$.

Without loss of generality, we may assume \(\varphi_3 = \{1,2,4\}\) and \(\psi_3 =\{2,4,5\}\).
Then, the other two chains, say $\omega, \tau \in \Omega\setminus\{\varphi, \psi\}$, should satisfy that \(\binom{\omega_3}{2}\cup\binom{\tau_3}{2}\supseteq\{\{3,1\},\{3,2\},\{3,4\},\{3,5\}\}\).
Then $\{\psi_3, \tau_3\}$ is one of the following:
\begin{center}
\begin{itemize*}
  \item $\{\{1,2,3\},\{3,4,5\}\}\;\;\;\;\;$
  \item $\{\{1,3,4\},\{2,3,5\}\}\;\;\;\;\;$
  \item $\{\{1,3,5\},\{2,3,4\}\}$.
\end{itemize*}
\end{center}
We observe that $\{\omega_3,\tau_3\}=\{\{1,3,5\},\{2,3,4\}\}$ since the other cases does not have \(\{1,5\}\) in $\text{code}(F)$.

So far, we have shown that $\{\omega^{i,j}:i,j\in[2]\} = \{\varphi,\psi,\omega,\tau\}$ with $$\{\varphi_3,\psi_3,\omega_3,\tau_3\}=\{\{1,2,4\},\{2,4,5\},\{1,3,5\},\{2,3,4\}\}.$$
However, for any choice, we have that $\binom{[5]}{2} \nsubseteq\text{code}(F)$.
Therefore, it must be \(N(2)\leq 5\).
\end{proof}

\section{Intersection patterns of $d$-tuplewise $S$-intersecting axis-parallel boxes}\label{sec:further}
As a remark, we prove a linear bound for an analogue of Theorem~\ref{thm:main2}, replacing the pairwise intersection pattern with a $d$-tuplewise condition.
\begin{theorem}\label{thm:main3}
    There exists a constant $C>0$ such that for every $S\subset\mathbb{R}^d$ and family $F$ of axis-parallel boxes in $\mathbb{R}^d$ with $|F|=Cd$, if every $d$-tuple of $F$ is $S$-intersecting, then $F$ contains an $S$-intersecting subfamily of size $d+1$.
\end{theorem}
The above statement can be achieved by proving an analogue of Theorem~\ref{thm:f(n,a,2)}: for every positive integers $a,b$ and $p$, there exists $n = n(a,b;p)$ such that \(d_p(A)\geq b-1\) for all \(A\subseteq S_{n}\) with \(|A|\leq a\).
Note that the function $n(a,b)$ in Theorem~\ref{thm:f(n,a,2)} is a special case that can be obtained as $n(a,b;2)$.
By Lemma~\ref{lem:main}, Theorem~\ref{thm:main3} can be obtained by showing that $n(2d,2;d)$ is linear in $d$.

For each permutation \(\sigma\in S_n\) and $v\in[n]$, we define \(\sigma\setminus v\in S_{n-1}\) as deleting and relabeling.
This is, $\sigma\setminus v$ is obtained by the following rules.
\begin{itemize}
    \item If \(i<\sigma^{-1}(v)\) and \(\sigma(i)<v\), then \((\sigma\setminus v)(i) = \sigma(i)\)
    \item If \(i<\sigma^{-1}(v)\) and \(\sigma(i)>v\), then \((\sigma\setminus v)(i) = \sigma(i)-1\).
    \item If \(i\geq\sigma^{-1}(v)\) and \(\sigma(i+1)<v\), then \((\sigma\setminus v)(i) = \sigma(i+1)\).
    \item If \(i\geq \sigma^{-1}(v)\) and \(\sigma(i+1)>v\), then \((\sigma\setminus v)(i) = \sigma(i+1)-1\).
\end{itemize}
For example, \(15672834\setminus 6 = 1562734\) and \(15672834\setminus 3 = 1456273\).

    Let $\sigma/v := (\sigma^{-1}\setminus v)^{-1}$, and for $V=\{v_1,\ldots,v_m\}\subset[n]$, let \[\sigma\setminus V = \{(((\sigma\setminus v_1)\setminus v_2)\cdots)\setminus v_m\}\text{ and }\sigma/V = (((\sigma/v_1)/v_2)\cdots)/v_m.\]
    For $A\subset S_n$, let $A\setminus V = \{\sigma\setminus V:\sigma\in A\}$ and $A/V = \{\sigma/V:\sigma\in A\}$.

    Define
    $$i^{v-}:=\begin{cases}i & \text{if }i<v\\i - 1 & \text{if }i >v\end{cases},\;\;\;i^{v+}:=\begin{cases}i & \text{if }i<v\\i + 1 & \text{if }i \geq v\end{cases},$$
    and for $P\subset[n-1]$ and $v\in[n-1]$, let $P^{v+} = \{i^{v+}: i\in P\}$. For $P\subset[n-m]$ and $V=\{v_1,\ldots,v_m\}\subset[n]$, let $P^{V+} = ((P^{v_1+})^{v_2+})^{\cdots})^{v_m+}$.
    This operation is indeed the order preserving bijection between \([n-m]\) and \([n]\setminus V\).

\begin{lemma}\label{lem:relabel}
    For every \(A\subseteq S_n\), $\sigma\in A$, \(P\subseteq[n-1]\), and $v\in[n]$, we have \(\langle A/v;P\rangle^{v+}\subseteq\langle A;P^{v+}\rangle\).
\end{lemma}
\begin{proof}
Suppose \(s\not=v\). Since $\sigma^{-1}(\sigma(s)) = s$ and $(\sigma^{-1}\setminus v)(\sigma(s)^{\sigma(v)-}) = s^{v-}$, we have $$(\sigma/v)(s^{v-})=\sigma(s)^{\sigma(v)-}.$$
Thus we observe that for every $s,t\in[n-1]$ with $s,t\neq v$, \(s^{v-}\prec_{\sigma/v}t^{v-}\) if and only if \(s\prec_\sigma t\). 
In particular, this implies \(\{s^{v+}: s\in (\sigma/v)(\leq P)\}\subseteq \sigma(\leq P^{v+})\), so we obtain \[\langle A/v;P\rangle^{v+}=\{s^{v+}\mid s\in \bigcap_{\sigma\in A}(\sigma/v)(\leq P)\}\subseteq\bigcap_{\sigma\in A}\sigma(\leq P^{v+})=\langle A;P^{v+}\rangle,\]
as desired.
\end{proof}
Now we prove a recurrence relation of $n(a,b;p)$ that implies that $n(2d,2;d)$ is linear in $d$.
\begin{theorem}\label{thm:p-dep}
    Given integers $a,b,p\geq1$, we have
    \[
        n(a+2,b;p+1)\leq n(a,b;p) + 2  \left\lceil\frac{n(a,b;p)}{a}\right\rceil - 1
    \]
\end{theorem}
\begin{proof}
Let \(n=n(a,b;p)\) and \(\Delta=\left\lceil\frac{n(a,b;p)}{a}\right\rceil\). Suppose \(A\subseteq S_{n+2\Delta-1}\) is a multiset with \(|A|=a+2\), say \(A=\{\sigma_0,\sigma_1,\cdots,\sigma_{a+1}\}\).
Considering \(\sigma_i^{-1}(j)\) for \(0\leq i\leq a + 1\) and \(n + \Delta\leq j\leq n+2\Delta-1\), we have \((a+2)\Delta>a\frac{n-1}{a}+2\Delta=n+2\Delta-1\), so there exist \(0\leq i\not=i'\leq a+1\) and \(n+\Delta\leq j, j'\leq n+2\Delta-1\) such that \(\sigma_{i}^{-1}(j)=\sigma_{i'}^{-1}(j')=k\).
Without loss of generality, we may assume \(i=a\) and \(i'=a+1\).

Let \(V=\{\sigma_{a}^{-1}(j),\sigma_{a+1}^{-1}(j)\mid n+\Delta\leq j\leq n+2\Delta-1\}\).
Then we have \(k\in V\) and \(|V|\leq 2\Delta-1\).
Since $n+2\Delta-1-|V| \geq n$, for \(A':=(A\setminus\{\sigma_a,\sigma_{a+1}\})/V\subseteq S_{n+2\Delta-1-|V|}\), by the definition of $n(a,b;p)$, there exists \(P\in\binom{[n+2\Delta-1-|V|]}{p}\) such that \(d( A'; P)\geq b-1\).

By a repeated application of Lemma~\ref{lem:relabel}, we have \(\langle A';P\rangle^{V+}\subseteq\langle A\setminus\{\sigma_a,\sigma_{a+1}\};P^{V+}\rangle\).
Also, since \(k\in V\), we have \(k\not\in \langle A';P\rangle^{V+}\).
From the definition of \(V\), for every \(v\in[n+2\Delta-1-|V|]\), \(\sigma_a(v^{V+}),\sigma_{a+1}(v^{V+})<n+\Delta\leq\sigma_a(k),\sigma_{a+1}(k)\), so \(\langle A';P\rangle^{V+}\prec_{\sigma_a} k\) and \(\langle A';P\rangle^{V+}\prec_{\sigma_{a+1}} k\).
Thus we have
\[
    \langle A';P\rangle^{V+}\mathbin{\mathaccent\cdot\cup} \{k\}
    \subseteq\langle A\setminus\{\sigma_a,\sigma_{a+1}\};P^{V+}\mathbin{\mathaccent\cdot\cup}\{k\}\rangle\cap \langle\{\sigma_a, \sigma_{a+1}\};\{k\}\rangle
    \subseteq \langle A;P^{V+}\mathbin{\mathaccent\cdot\cup}\{k\}\rangle,
\]
which proves
\begin{align*}
d_{p+1}(A) & \geq d(A;P^{V+}\mathbin{\mathaccent\cdot\cup}\{k\})
\\& \geq|\langle A';P\rangle^{V+}\mathbin{\mathaccent\cdot\cup} \{k\}|-|P^{V+}\mathbin{\mathaccent\cdot\cup}\{k\}| \\&=(|\langle A';P\rangle| + 1) - (|P| + 1)
\\&= d(A';P)\geq b-1.
\end{align*}
This completes the proof.
\end{proof}

\begin{corollary}\label{cor:p-dep}
    Given integers \(a\geq 1,b\geq 2\), we have $n(a+2p-4,b;p)\in O(p)$
\end{corollary}
\begin{proof}
If \(\Delta=\left\lceil\frac{n}{a}\right\rceil\), then we have \[
\frac{n+2\Delta-1}{a+2}\leq\frac{a\Delta+2\Delta-1}{a+2}<\Delta,
\]
which implies
\[
n(a+2p-4,b;p)\leq n(a,b;2) + \left(2\left\lceil\frac{n(a,b;2)}{a}\right\rceil-1\right)(p-2)\in O(p),
\]
as desired.
\end{proof}
\begin{proof}[Proof of Theorem~\ref{thm:main3}]
    This is a direct application of Lemma~\ref{lem:main} and Corollary~\ref{cor:p-dep} by setting $a=4$ and $b=2$.
    Note that $n(4,2;2)=n(4,2)$ exists by Theorem~\ref{thm:f(n,a,2)}.
\end{proof}

\bibliographystyle{acm}
\bibliography{reference.bib}

\begin{thebibliography}{1}

\bibitem{ALS17}
{\sc Amenta, N., De~Loera, J.~A., and Sober\'on, P.}
\newblock Helly's theorem: new variations and applications.
\newblock In {\em Algebraic and geometric methods in discrete mathematics}, vol.~685 of {\em Contemp. Math.} Amer. Math. Soc., Providence, RI, 2017, pp.~55--95.

\bibitem{BCH+24}
{\sc Benitez, M., Chen, S., Han, T., Jeffs, R.~A., Paguyo, K., and Zhou, K.~A.}
\newblock Realizing convex codes with axis-parallel boxes.
\newblock {\em Involve 17}, 4 (2024), 633--649.

\bibitem{Eck88}
{\sc Eckhoff, J.}
\newblock Intersection properties of boxes. i . an upper-bound theorem.
\newblock {\em Israel Journal of Mathematics 62}, 3 (1988), 283–301.

\bibitem{ES24+}
{\sc Edwards, T., and Soberón, P.}
\newblock Extensions of discrete helly theorems for boxes.
\newblock {\em preprint\/} (2024), arXiv:2404.14308.

\bibitem{Hal08}
{\sc Halman, N.}
\newblock Discrete and lexicographic helly-type theorems.
\newblock {\em Discrete \& Computational Geometry 39}, 4 (2008), 633–649.

\bibitem{Kal84}
{\sc Kalai, G.}
\newblock Intersection patterns of convex sets.
\newblock {\em Israel Journal of Mathematics 48}, 2-3 (1984), 161–174.

\end{thebibliography}

\end{document}